\documentclass[12pt]{amsart}

\usepackage{tgpagella}
\usepackage{euler}
\usepackage[T1]{fontenc}
\usepackage{amsmath, amssymb}
\usepackage[hidelinks]{hyperref}
\usepackage[english]{babel}
\usepackage{mathrsfs}
\usepackage{eucal}
\usepackage[all]{xy}
\usepackage{tikz}
\usepackage{stmaryrd}

\usepackage{todonotes}

\newtheorem{thm}{Theorem}[section]
\newtheorem*{thm*}{Theorem}
\newtheorem{lem}[thm]{Lemma}
\newtheorem{fact}[thm]{Fact}

\newtheorem{prop}[thm]{Proposition}
\newtheorem*{prop*}{Proposition}

\newtheorem{cor}[thm]{Corollary}
\newtheorem*{cor*}{Corollary}

\theoremstyle{definition}
\newtheorem{defn}[thm]{Definition}
\newtheorem*{defn*}{Definition}

\newtheorem{question}[thm]{Question}
\newtheorem{example}[thm]{Example}
\newtheorem{examples}[thm]{Examples}

\newtheorem*{question*}{Question}
\newtheorem*{Pquestion*}{Popa's question}

\newtheorem*{conv*}{Convention}

\newcommand{\dminus}{ 
\buildrel\textstyle\ .\over{\hbox{ 
\vrule height3pt depth0pt width0pt}{\smash-} 
}}

\def\bb{\mathbb}

\def\de{\delta}

\def\Om{\Omega}
\def\bb{\mathbb}

\def\cal{\mathcal}

\def\BorI{\mathcal{B}([0,1))}

\makeatletter

\def\dotminussym#1#2{%
  \setbox0=\hbox{$\m@th#1-$}%
  \kern.5\wd0%
  \hbox to 0pt{\hss\hbox{$\m@th#1-$}\hss}%
  \raise.6\ht0\hbox to 0pt{\hss$\m@th#1.$\hss}%
  \kern.5\wd0}

\def \mbor{\mathcal{M}^{[0,1)}}

\title{Computable presentations of randomizations}
\author{Nicol\'as Cuervo Ovalle and Isaac Goldbring}

\address{Department of Mathematics\\ Los Andes University, Bogotá, Colombia,
Cra. 1 \#18a-12}
\email{n.cuervo10@uniandes.edu.co}
\thanks{The first named author was partially supported by NSF grant DMS-2054477. He would also like to thank the UC Irvine Department of Mathematics for their hospitality.}

\address{Department of Mathematics\\University of California, Irvine, 340 Rowland Hall (Bldg.\# 400),
Irvine, CA 92697-3875}
\email{isaac@math.uci.edu}
\urladdr{http://www.math.uci.edu/~isaac}
\thanks{The second-named author was partially supported by NSF grant DMS-2054477.}

\begin{document}

 \begin{abstract}
We initiate the effective metric structure theory of Keisler randomizations.  We show that a classical countable structure $\cal M$ has a decidable presentation if and only if its Borel randomization $\mbor$ has a computable presentation for which the constant functions are uniformly computable points.  We determine a sufficient condition for which the uniform computability of the constant functions can be dropped.  We show that when $\cal M$ is effectively $\omega$-categorical, then $\mbor$ is computably categorical, that is, has a unique computable presentation up to computable isomorphism.  A special case of this result is that the unique separable atomless probability algebra is computably categorical.  Finally, we show that all randomizations admit effective quantifier elimination.
 \end{abstract}
\maketitle

\section{Introduction}

In \cite{keisler}, Keisler introduced the notion of randomizing a (classical) structure $\cal M$, which is a structure whose elements are viewed as $\cal M$-valued random variables.  Keisler's initial construction was presented in classical logic; the construction was later transported by Ben Yaacov and Keisler \cite{BYK} into the context of continuous logic.  In \cite{benyaacov}, Ben Yaacov showed how to further generalize the construction to allow for randomizations of continuous theories as well.  

Randomizations are known to preserve many important model-theoretic phenomena, such as countable categoricity, stability (but not simplicity), $\omega$-stability, NIP, etc... (see \cite{benyaacov2} for the last statement).  In this article, we investigate to what extent the randomization construction interacts nicely with ideas from \emph{computable structure theory}.  More precisely, we investigate the question as to whether or not there is a connection between a (countable, classical) structure $\mathcal{M}$ having a \emph{computable presentation} and its \emph{Borel randomization} $\mbor$ having a computable presentation.  We now roughly define the italicized terms in the previous sentence; precise definitions will be given in the next section.

Given a countable (classical) structure $\mathcal{M}$, a presentation of $\mathcal{M}$ is a function $\nu:\bb N\to \cal M$, which we think of as a labeling of the elements of $\mathcal{M}$ (perhaps with repetitions).  The presentation is computable if, roughly speaking, the operations and predicates on $\mathcal{M}$ lift to computable functions and subsets of $\bb N$.  An analogous definition of presentation can be given to a separable metric structure, with the difference being that the presentation labels a countable set that generates a dense subset of the structure; such points are said to be the \emph{special points} of the presentation while the points in the pre-substructure that they generate are said to be \emph{generated points}.  To say that this presentation is computable amounts to the fact that (again, roughly speaking) one can compute generated points that approximate operations applied to generated points (and a similar statement for predicates).

Given a classical countable structure $\mathcal{M}$, its \emph{Borel randomization} $\mbor$ is the set of function $f:[0,1)\to \mathcal{M}$ such that, for each $a\in \mathcal{M}$, we have that $f^{-1}(a)$ is a Borel subset of $[0,1)$.  We consider the Borel randomization  as a structure (technically speaking, pre-structure) in a two-sorted continuous language, with one sort for the $\mathcal{M}$-valued random variables in the previous sentence and another for the probability algebra of $[0,1)$; for each first-order formula $\varphi(\vec x)$ with $\vec x=(x_1,\ldots,x_n)$, there is a function symbol $\llbracket \varphi(\cdot)\rrbracket$ which maps a tuple $\vec f$ of random variables to the event $\llbracket \varphi(\vec f)\rrbracket:=\{\omega \in [0,1) \ : \ \mathcal{M}\models \varphi(\vec f(\omega))\}$.  The pseudometric on the random variable sorts is given by $d(f,g):=\mu\llbracket f\not=g\rrbracket$ and the pseudometric on the probability algebra is the usual one given by $d(A,B):=\mu(A\triangle B)$.  The completion of this pre-structure is then a metric structure, which we also refer to as $\mbor$.

The na\"ive guess might be that $\mathcal{M}$ has a computable presentation if and only if $\mbor$ has a computable presentation.\footnote{Throughout this introduction, we assume that $\cal M$ is a structure in a computable language; this assumption will be stated explicitly in the next section as well.}  However, such a na\"ive statement presents two issues.  First of all, since a computable presentation of $\mbor$ can approximate the measure of the event $\llbracket \varphi(\vec f)\rrbracket$, where $\varphi(\vec x)$ is any formula (which  may have quantifiers), it seems likely that any conclusion drawn about a presentation of $\mathcal{M}$ would show that it is not only computable, but rather \emph{decidable}, which means that one can effectively determine the truth of any formula with parameters in $\mathcal{M}$ (and not just the quantifier-free ones, whose truth can be verified relative to a computable presentation).  Since there are classical structures that admit a computable presentation but no decidable presentation, it seems more sensible to ask if $\mathcal{M}$ has a decidable presentation if and only if $\mbor$ has a computable presentation.  Another (but related) reason for changing the question is that any computable presentation of a randomization is automatically decidable.  The reason for this is that randomizations have effective quantifier elimination (see Section \ref{appendix1}), which essentially follows from the atomic formulae in the randomization language having classifical quantifiers embedded in them. 

Considering the modified question, in one direction, given a decidable presentation of $\mathcal{M}$, we show how to construct a natural ``induced'' presentation on $\mbor$, which we then show is computable.  However, in the converse direction, it is not clear to us how to extract a presentation (decidable or otherwise) of $\mathcal{M}$ from an arbitrary presentation of $\mbor$.  That being said, if we somehow have ``computable access'' to the constant functions in $\mbor$, then we can indeed recover a decidable presentation of $\mathcal{M}$.  Our main result can thus roughly be stated as follows:

\begin{thm*}
For any classical, countable structure $\mathcal{M}$, we have that $\mathcal{M}$ has a decidable presentation if and only if $\mbor$ has a computable presentation for which the constant functions are uniformly computable points.
\end{thm*}

In general, we do not know if the previous theorem is valid without the assumption that the constant functions are uniformly computable points.  However, we will present a condition on a structure $\mathcal{M}$, called being \emph{effectively recognizable}, for which we can remove this assumption.

\begin{cor*}
Suppose that $\mathcal{M}$ is an effectively recognizable structure.  Then $\mathcal{M}$ has a decidable presentation if and only if $\mbor$ has a computable presentation. 
\end{cor*}

For any structure $\mathcal{M}$, its expansion $\mathcal{M}_M$ obtained by naming the elements of $\mathcal{M}$ is automatically effectively recognizable.  Since $\mathcal{M}$ has a decidable presentation if and only if $\mathcal{M}_M$ has a decidable presentation, we have:

\begin{cor*}
For any classical, countable structure $\mathcal{M}$, we have that $\mathcal{M}$ has a decidable presentation if and only if $\mathcal{M}_M^{[0,1)}$ has a computable presentation.
\end{cor*}

As mentioned above, the notions of ``computable presentation'' and ``decidable presenation'' coincide for $\mbor$, whence one may obtain more symmetric statements in the previous two results by replacing ``computable'' by ``decidable.''

We also consider the question as to when $\mbor$ is \emph{computably categorical}, that is, when $\mbor$ has a unique computable presentation (up to computable isomorphism).  In Section \ref{appendix2}, we establish the following:

\begin{thm*}
If $\cal M$ is \emph{effectively $\omega$-categorical, then $\mbor$ is computably categorical.  }
\end{thm*}

For the definition of effectively $\omega$-categorical structure, see Definition \ref{effomega}.  The previous theorem includes the special case that $\cal M$ is the two-element set with no further structure, whence we conclude that the unique atomless probability algebra is computably categorical, a result of interest in its own right.\footnote{In fact, we prove this special case first as motivation for the more general argument.}

In the final section, we show that any randomization admits effective quantifier-elimination.  Quantifier-elimination for randomizations was first proven by Keisler in \cite{keisler} in his original presentation of randomizations as structures in first-order logic.  In \cite{BYK}, where randomizations were treated as metric structures, quantifier-elimination was simply claimed without proof (see \cite[Theorem 2.9]{BYK}), and a reference to Keisler's original paper \cite{keisler} was quoted.  Here, we elaborate on how one might carry out Keisler's proof of quantifier-elimination in the continuous setting, while paying attention to effective matters.  Our proof also seems to necessitate the appropriate notion of a definable family of definable subsets of a metric structure, the basics of which we elaborate in the course of proving the theorem.  Since effective quantifier-elimination in continuous logic seems to not have appeared in the literature before, we begin Section \ref{appendix1} by developing some basics facts about the notion.

We assume that the reader is familiar with basic continuous logic; the canonical reference is \cite{BBHU}.  It would help if the reader was familiar with basic effective metric structure theory, although we develop the amount of this material that we need in the following section.

\section{Preliminaries}

\subsection{Randomizations}

Fix a countable first-order language $L$ and a countable $L$-structure $\mathcal{M}$; for simplicity, let us assume that $\mathcal{M}$ is one-sorted.  We define a two-sorted, continuous pre-structure $\mbor$ as follows:
\begin{itemize}
    \item The first sort, which we denote by $\mathcal{K}$, is the set of functions $f:[0,1)\to \mathcal{M}$ such that, for each $a\in \mathcal{M}$, we have that $f^{-1}(\{a\})$ is a Borel subset of $[0,1)$.  We think of such a function as an $\mathcal{M}$-valued random variable.
    \item The second sort is the probability algebra $\mathcal{B}$ associated to the probability space $[0,1)$ equipped with its Lebesgue measure.  We treat this sort as a structure in the usual continuous language for probability algebras, that is, with function symbols for the Boolean operations and a predicate symbol for Lebesgue measure.
    \item For each $L$-formula $\varphi(x_1,\ldots,x_n)$, we include a distinguished function $\mathcal{K}^n\to \mathcal{B}$ given by $$\llbracket\varphi(f_1,\ldots,f_n)\rrbracket:=\{\omega\in [0,1) \ : \ \mathcal{M}\models \varphi(f_1(\omega),\ldots,f_n(\omega)\}.$$
\end{itemize}

We let $L^R$ denote the language for which $\mbor$ is a $L^R$-pre-structure.  We refer to $\mbor$ as the \textbf{Borel randomization of $\mathcal{M}$.}  We will also use $\mbor$ to denote the completed structure and $\cal K$ to denote the corresponding sort of the completion.  We hope that this does not cause the reader any confusion.  \emph{If we have the need to specify that an element of sort $\cal K$ from $\mbor$ belongs to the pre-structure described above, we will say that that element is a \textbf{random variable} from $\mbor$.}  

We will need to recall one important fact about $\mbor$:

\begin{fact}\label{fullness}
    For every $L$-formula $\varphi(x_1,\ldots,x_n,y)$, every $f_1,\ldots,f_n\in \mathcal{K}$, and every $\epsilon>0$, there is $g\in \mathcal{K}$ such that $$\mu(\llbracket \exists y\varphi(f_1,\ldots,f_n,y)\rrbracket\triangle \llbracket\varphi(f_1,\ldots,f_n,g)\rrbracket)<\epsilon.$$
\end{fact}
In the sequel, we abuse terminology and view $\mathcal{M}$ as a subset of $\mbor$ by identifying each element of $\mathcal{M}$ with the associated constant random variable in $\mathcal{K}$.

We borrow terminology from measure theory and write $\sum_{i=1}^n a_i\chi_{I_i}$ for the random variable which, for each $i=1,\ldots,n$, is constantly $a_i$ on the Borel set $I_i$ (and where we implicitly assume that the $I_i$'s are pairwise disjoint); we call such an element of $\mathcal{K}$ a \textbf{simple function}.

\subsection{Presentations of structures}

Suppose that $\mathcal{N}$ is a separable metric structure, which, for simplicity, we momentarily assume is one-sorted and of diameter at most $1$.  By a \textbf{presentation} of  $\mathcal{N}$ we mean a countable sequence $(a_j)_{j\in \mathbb{N}}$ from $\mathcal{N}$ that generates $\mathcal{N}$ (that is, the closure of the sequence under the interpretations of function symbols is dense in $\mathcal{N}$).  We may denote a presentation of $\mathcal{N}$ as a pair $\mathcal{N}^{\#}:=(\mathcal{N},(a_j)_{j\in \mathbb{N}})$.  The points $a_j$ are called \textbf{special points of the presentation} and the points in the closure of the special points under the function symbols are called \textbf{generated points of the presentation}.  The definition in the many-sorted case is analogous.  Moreover, we extend the terminology to cover the case of pre-structures (where the metric need not be complete).

Note that the previous definition applies also to the case that $\mathcal{N}$ is a countable classical structure treated as a continuous structure by equipping it with the discrete metric.  In this case, the generated points of the presentation comprise all of $\mathcal{N}$.  

\emph{From now on, when considering presentations of metric structures, we implicitly assume that the signature for $\mathcal{N}$ is computable, whence one can produce an effective enumeration of the generated points of a presentation $\mathcal{N}^{\#}$ of $\mathcal{N}$.}

In connection with the convention just established, note that if $L$ is a computable classical language, then the corresponding randomization language $L^R$ is also computable.

A presentation $\mathcal{N}^\#$ of $\mathcal{N}$ is called \textbf{computable} if there is an algorithm for which, upon input an $n$-ary predicate symbol $P$ of the language, an $n$-tuple $\vec a$ of generated points of $\mathcal{N}^\#$, and rational $\epsilon>0$, returns a real number $r$ such that $|P^{\mathcal{N}}(\vec a)-r|<\epsilon$.

In connection with presentations of randomizations in the next section, we note that the probability algebra associated to $[0,1)$, 
which we denote by $\BorI$, has a ``standard presentation'' which consists of all of open intervals with rational endpoints\footnote{We include an interval of the form $[0,\epsilon)$ as an open interval in $[0,1)$.}; it is clear that this presentation is computable.

\begin{defn}
    Suppose that $\mathcal{N}_1$ and $\mathcal{N}_2$ are two structures with presentations $\mathcal{N}_1^\#$ and $\mathcal{N}_2^\#$ respectively.  An embedding $f:\mathcal{N}_1\to \mathcal{N}_2$ is called \textbf{a computable map from $\mathcal{N}_1^\#$ to $\mathcal{N}_2^\#$} if there is an algorithm such that, upon input a generated point $p$ of $\mathcal{N}_1^\#$ and rational $\epsilon>0$, returns a generated point $q$ of $\mathcal{N}_2^\#$ such that $d(f(p),q)<\epsilon$.\footnote{This is not the official definition of a computable map between presented structures, but is equivalent to the usual definition in our context; see \cite[Section 3.5]{EMST}.}  If, in addition, $f$ is invertible and $f^{-1}$ is a computable map from $\mathcal{N}_2^\#$ to $\mathcal{N}_1^\#$, then we say that $f$ is a \textbf{computable isomorphism}; when a computable isomorphism exists between $\mathcal{N}_1^\#$ and $\mathcal{N}_2^\#$ we say that $\mathcal{N}_1^\#$ and $\mathcal{N}_2^\#$ are \textbf{computably isomorphic}.  
\end{defn}

\begin{defn}
    Suppose that $\mathcal{N}^\#$ is a presentation of $\mathcal{N}$. 
    \begin{enumerate}
        \item A sequence $(p_n)_{n\in \mathbb{N}}$ from $\mathcal{N}$ is called a \textbf{computable sequence of $\mathcal{N}^\#$ } if there is an algorithm such that, upon input $n\in \mathbb{N}$ and rational $\epsilon>0$, returns a generated point $q$ of $\mathcal{N}^\#$ such that $d(p_n,q)<\epsilon$.
        \item A closed subset $X\subseteq \mathcal{N}$ is a \textbf{c.e. closed subset of $\mathcal{N}^\#$} if the set of open balls of the form $B(p;\epsilon)$, with $p$ a generated point and $\epsilon>0$ a rational number, that intersect $X$ is c.e.
    \end{enumerate}
\end{defn}
An open ball as in the second item of the previous definition is called a \textbf{rational open ball of $\mathcal{N}^\#$}.  If $U=B(p;\epsilon)$ and $V=B(q;\delta)$ are two such rational open balls of $\mathcal{N}^\#$, we say that $U$ is \textbf{formally included in} $V$ if $d(p,q)+\epsilon<\delta$.  

Let us call a presentation $\mathcal{N}^\#$ of $\mathcal{N}$ \textbf{weakly computable} if there is an algorithm such that, upon inputs generated points $p$ and $q$ of $\mathcal{N}^\#$and rational $\epsilon>0$, returns a rational number $r$ such that $|d(p,q)-r|<\epsilon$.  Note that in case that $\mathcal{N}$ is a classical structure equipped with the discrete metric, saying that $\mathcal{N}^\#$ is weakly computable is equivalent to saying that the equality relation on the generated points is computable.  Returning to our discussion of rational open balls, we note that if $\mathcal{N}^\#$ is weakly computable, then the formal inclusion relation between rational open balls is c.e.

Following \cite[Section 6]{BBHU}, the continuous functions generated by the unary functions $0$, $1$, and $\frac{x}{2}$ and the binary function $x\dminus y$ by closing under composition are called \textbf{restricted}.  A formula is called restricted if all connectives used in its formation are restricted.

If $\mathcal{N}^\#$ is a computable presentation, then it follows that there is an algorithm such that, upon input a restricted quantifier-free formula $\varphi(\vec x)$, a tuple $\vec f$ of generated points of $\mathcal{N}^\#$, and rational $\epsilon>0$, returns a rational number $r$ such that $|\varphi(\vec f)^{\mathcal{N}}-r|<\epsilon$.  If $\mathcal{N}^\#$ satisfies the assertion of the previous sentence for all restricted formulae (and not just the quantifier-free ones), then we say that $\mathcal{N}^\#$ is a \textbf{decidable presentation of $\mathcal{N}$}.  In case that $\mathcal{N}$ is a classical structure with the discrete metric, this is equivalent to saying that one may decide whether or not $\mathcal{N}\models \varphi(\vec f)$ for any formula $\varphi(\vec x)$ and any generated tuple $\vec f$ of $\mathcal{N}^\#$.

\subsection{Presentations of randomizations}

Fix a presentation $\mathcal M^\#$ of $\mathcal M$.  We obtain an \textbf{induced presentation} $(\mbor)^\#$ of $\mbor$  whose $i^{\text{th}}$ special point is the random variable which is the simple function $\sum_{j=1}^n a_j\chi_{I_j}$, where $i$ is the code for the finite sequence $(i_1,\ldots,i_n,J_1,\ldots,J_n)$, $a_j$ is the $i_j^{\text{th}}$ generated point of $\mathcal{M}^\#$, and where $J_k$ is the $k^{\text{th}}$ generated point of the standard presentation of $\BorI$.  Note that this is indeed a presentation of $\mbor$ as such simple functions are dense in $\mathcal{K}$ and all generated elements of the standard presentation of $\BorI$ are easily seen to be generated elements of $(\mbor)^\#$; in fact, the presentation on $\BorI$ induced by $(\mbor)^\#$ is (computably isomorphic to) the standard presentation of $\BorI$. Note also that, since there are no function symbols with co-domain $\mathcal{K}$, the special points and generated points in $\mathcal{K}$ coincide; in fact, this comment holds for any presentation of $\mbor$.  Finally note that the notions of computable presentation and weakly computable presentation coincide for randomizations as the only predicate symbol $\mu$ is defined in terms of the distance to the bottom element of the probability algebra.


\begin{defn}
We say that a presentation $(\mbor)^\#$ of $\mbor$ is \textbf{aware} if $\cal M$ is a c.e. closed set of $(\mbor)^\#$.
\end{defn}

\begin{lem}\label{inducedaware}
If $\cal M^\#$ is a weakly computable presentation of $\cal M$, then the induced presentation $(\mbor)^\#$ of $\mbor$ is aware.
\end{lem}

\begin{proof}
Given a special point $f=\sum_{j=1}^n a_j \chi_{I_j}$ of $(\mbor)^\#$ and rational $\epsilon>0$, to see if $B(f;\epsilon)$ intersects $\cal M$, one needs to check if there are indices $j_1,\ldots,j_t$ such that $a_{j_1},\ldots,a_{j_t}$ coincide and for which $\sum_{p=1}^t \mu(I_{j_p})>\epsilon$; this can be done since the presentation $\cal M^\#$ is weakly computable.
\end{proof}

\begin{prop}
    For an arbitrary computable presentation $(\mbor)^\#$ of $\mbor$, the following are equivalent:
    \begin{enumerate}
    \item  $(\mbor)^\#$ is aware.
        \item There is an enumeration $(a_n)_{n\in \bb N}$ of $\mathcal{M}$ that is a computable sequence of $(\mbor)^\#$.
    \end{enumerate}
\end{prop}
\begin{proof}
First suppose that $\mathcal{M}$ is a c.e. closed set of $(\mbor)^\#$.   Since $\mathcal{M}$ is a discrete subset of $\mathcal{K}$, it is closed.  Now suppose that $(U_n)_{n\in \mathbb{N}}$ is a c.e. enumeration of all of the rational open balls of $(\mbor)^\#$ of radius less than $1/2$ that intersect $\mathcal{M}$.  Let $a_n$ be the unique element of $\mathcal{M}$ in $U_n$.  Then $(a_n)_{n\in \mathbb{N}}$ is an enumeration of $\mathcal{M}$.  To see that $(a_n)_{n\in \mathbb{N}}$ is a computable sequence of $(\mbor)^\#$, it suffices to note that, given $n\in \mathbb{N}$ and $\epsilon>0$, one may search for the first $m$ such that $U_m$ is formally included in $U_n$ and whose radius is less than $\epsilon$; it follows that the center $q$ of $U_m$ is a generated point of $(\mbor)^\#$ with $d(a_n,q)<\epsilon$. 

Conversely, suppose that there is an enumeration $(a_n)_{n\in \mathbb{N}}$ of $\mathcal{M}$ that is a computable sequence of $(\mbor)^\#$.  To computably enumerate the rational open balls that intersect $\mathcal{M}$, it suffices to list, for each $n\in \mathbb{N}$ and rational $\epsilon>0$, a generated point $p_{n,\epsilon}$ of $(\mbor)^\#$ such that $d(a_n,p_{n,\epsilon})<\epsilon$, and then to include in the enumeration all rational open balls that formally include some $B(p_{n,\epsilon};\epsilon)$.
\end{proof}



\begin{defn}
     If $(\mbor)^\#$ is an aware presentation of $\mbor$, we define the \textbf{induced presentation} $\mathcal M^{(\#)}$ of $\mathcal M$ as follows:  let $(B(f_i;r_i))_{i\in \bb N}$ be an effective listing of the rational open balls of $(\mbor)^\#$ which intersect the constant functions and for which $r_i<1/2$.  In this case, $B(f_i;r_i)$ contains a unique constant random variable, which we declare to be the $i^{\text{th}}$ special point of $\mathcal{M}^{(\#)}$.
\end{defn}

In the context of the previous definition, note that the induced presentation $\cal M^{(\#)}$ of $\cal M$ is such that the presentation is an actual enumeration of $\cal M$, whence it is automatically weakly computable.

If one starts with a weakly computable presentation $\mathcal{M}^\#$ of $\mathcal M$, then passes to the induced presentation $(\mbor)^\#$ of $\mbor$, which is aware by Lemma \ref{inducedaware}, and then passes to the induced presentation $\mathcal M^{(\#)}$ on $\mathcal M$, it is natural to compare $\mathcal{M}^\#$ and $\mathcal{M}^{(\#)}$.  The proof of the following lemma is left to the reader.

\begin{lem}
    In the above notation, the presentations $\mathcal{M}^\#$ and $\mathcal{M}^{(\#)}$ are computably isomorphic.
\end{lem}

Now suppose that one starts with a computable aware presentation $(\mbor)^\#$ of $\mathcal{M}$, then passes to the induced presentation $\mathcal{M}^{(\#)}$ of $\mathcal{M}$, and then to the corresponding induced presentation $(\mbor)^{(\#)}$ of $\mbor$; it is once again natural to compare $(\mbor)^\#$ and $(\mbor)^{(\#)}$ and we will again see that these presentation are computably isomorphic.

First, note that, for any automorphism $\sigma$ of $\BorI$ there is an induced automorphism, also denoted by $\sigma$, on $\mbor$, given on random variables $f$ by $\sigma(f)(\omega):=f(\sigma^{-1}(\omega))$, where $\sigma$ is any measure-preserving transformation of $[0,1)$ that induces $\sigma$ on $\BorI$.  If $\varphi(x_1,\ldots,x_n)$ is an $L$-formula and $f_1,\ldots,f_n$ are elements of $\mathcal{K}$,  we clearly have $$\llbracket \varphi(\sigma(f_1),\ldots,\sigma(f_n))\rrbracket =\sigma \llbracket \varphi(f_1,\ldots,f_n)\rrbracket.$$  If $(\mbor)^\#$ is a presentation of $\mbor$, then we obtain a new presentation $(\mbor)^{\#,\sigma}$ of $\mbor$ whose special points are those obtained by applying $\sigma$ to the special points of $(\mbor)^\#$.   Note that $\sigma$ is a computable isomorphism from $(\mbor)^\#$ to $(\mbor)^{\#,\sigma}$ if the original automorphism $\sigma$ on the probability algebra was computable.  The presentation $\BorI^{\#,\sigma}$ of $\BorI$ induced by $(\mbor)^{\#,\sigma}$ is that obtained by applying $\sigma$ to the special points of $\BorI^\#$.  By computable categoricity of $\BorI$ (see Section \ref{appendix2}), we obtain:

\begin{prop}\label{wlog}
  Any aware presentation of $\mbor$ is computably isomorphic to one that induces the standard presentation on $\BorI$.
\end{prop}    \begin{prop}\label{awarewlog}
        Suppose that $(\mbor)^\#$ is a computable aware presentation of  $\mbor$.  Let $\cal M^{(\#)}$ be the induced presentation of $\cal M$ which then induces the presentation $(\mbor)^{(\#)}$ on $\mbor$.  Then $(\mbor)^\#$ is computably isomorphic to $(\mbor)^{(\#)}$. Consequently, any computable aware  presentation of $\mbor$ is computably isomorphic to the induced presentation on $\mbor$ from a presentation of $\mathcal{M}$ (which will later be proven to be a decidable presentation of $\mathcal{M})$.
         \end{prop}

    \begin{proof}
        By Proposition \ref{wlog}, we may assume that the induced presentation on $\BorI$ is the standard presentation.  Since the metric on the special points of $(\mbor)^\#$ is computable, it suffices to show that each special point of $(\mbor)^{(\#)}$ is uniformly a computable point of $(\mbor)^\#$ (for then the identiy map on $\mbor$ will be a computable isomorphism from $(\mbor)^{(\#)}$ to $(\mbor)^\#$).  So consider a special point $f$ of $(\mbor)^{(\#)}$, which is a simple function $\sum_i a_i \chi_{I_i}$, where $I_i$ is an interval in $[0,1)$ with rational endpoints.  By the definition of $\mathcal{M}^{(\#)}$, we may search for special points $f_i$ of $(\mbor)^\#$ such that $d(a_i,f_i)$ is sufficiently small.  We then search for a special point $g$ of $(\mbor)^\#$ for which $d(\llbracket g=f_i\rrbracket, I_i)$ is sufficiently close to $\mu(I_i)$ for each $i=1,\ldots,n$.  It follows that $g$ is a special point of $(\mbor)^\#$ close to $f$, as desired.
    \end{proof}

We suspect that the following question has a negative answer, but we have been unable to provide a counterexample:
\begin{question}\label{awarequestion}
    Are all computable presentations of $\mbor$ aware?
    \end{question}

\section{Computable presentations of randomizations}

Until further notice, we fix a presentation $\cal M^\#$ of $\cal M$ and let $(\mbor)^\#$ be the induced randomization of $\mbor$.  If $\varphi(\vec x)$ is an $L$-formula and $\mathbf{d}$ is a Turing degree, let us say that $\varphi(\vec x)$ is \textbf{$\mathbf{d}$-computable with respect to $(\mbor)^\#$} if there is an algorithm with  oracle $\mathbf{d}$ such that, upon inputs generated points $\vec f$ of $(\mbor)^\#$ and $\epsilon>0$, returns a rational number $q$ such that $|\mu\llbracket \varphi(\vec f)\rrbracket-q|<\epsilon$.  If the algorithm instead merely returns a non-decreasing sequence of rational lower bounds that converge to $\mu\llbracket \varphi(\vec f)\rrbracket$, we say that $\varphi(\vec x)$ is \textbf{$\mathbf{d}$-left c.e. with respect to $(\mbor)^\#$}.  Note that if $\varphi(\vec x)$ is $\mathbf{d}$-left c.e. with respect to $(\mbor)^\#$, then it is $\mathbf{d}'$-computable with respect to $(\mbor)^\#$.

\begin{prop}\label{fullnesslemma}
Suppose that $\varphi(x,\vec y)$ is an $L$-formula that  is $\bf d$-computable with respect to $(\mbor)^\#$.  Then $\exists x\varphi(x,\vec y)$ is $\mathbf{d}$-left c.e. with respect to $(\mbor)^\#$.
\end{prop}

\begin{proof}
By Fact \ref{fullness}, the events $\llbracket \varphi(f,\vec g)\rrbracket$ approximate $\llbracket\exists x\varphi(x,\vec g)\rrbracket$ arbitrarily well from within; the result now follows from the assumption that $\varphi(x,\vec y)$ is $\mathbf{d}$-computable with respect to $(\mbor)^\#$.
\end{proof}

\begin{prop}\label{computingformulae}
If $\cal M^\#$ is a computable presentation of $\mathcal{M}$, then any  quantifier-free $L$-formula $\varphi(\vec x)$ is computable with respect to $(\mbor)^\#$.  Moreover, an algorithm witnessing that $\varphi(\vec x)$ is computable with respect to $(\mbor)^\#$ can be obtained from an algorithm witnessing that $\mathcal{M}^\#$ is a computable presentation of $\mathcal{M}$.
\end{prop}
\begin{proof}
  Suppose $\vec x=(x_1,\ldots,x_n)$ and consider a tuple $\vec f=(f_1,\ldots,f_n)$ of special points of $(\mbor)^\#$.    Write $f_i=\sum_j a_{ij}\chi_{I_{ij}}$.  Set $$X:=\{(j(1),\ldots,j(n)) \ : \ \mathcal M\models \varphi(a_{1j(1)},\ldots,a_{nj(n)}\};$$ since $\mathcal M^\#$ is computable, we see that $X$ can be computed from $\varphi$ and $\vec f$; it then follows that $\mu\llbracket\varphi(\vec f)\rrbracket=\sum_{(j(1),\ldots,j(n))\in X}\mu(I_{1j(1)}\cap \cdots\cap I_{nj(n)})$, which can then be effectively approximated.
\end{proof}

\begin{cor}\label{maincor}
If $\cal M^\#$ is computable, then $(\mbor)^\#$ is $\mathbf{0}^{(\omega)}$-computable.
\end{cor}
\begin{proof}
    This follows immediately from Propositions \ref{fullnesslemma} and \ref{computingformulae}. 
\end{proof}
In Proposition \ref{computingformulae}, if we assume that $\mathcal{M}^\#$ is a decidable presentation of $\mathcal{M}$, then the same proof shows that all $L$-formulae are computable with respect to $(\mbor)^\#$ and in a uniform manner.  In other words, we have the following improvement of Corollary \ref{maincor}:

\begin{cor}\label{maincor2}
If $\cal M^\#$ is decidable, then $(\mbor)^\#$ is computable.
\end{cor}
As stated in the introduction, since $\mbor$ has effective quantifier-elimination (see Section \ref{appendix1}), there is no difference between a computable presentation of $\mbor$ and a decidable one.  Consequently, the previous corollary could be restated in a more symmetric fashion:  if $\mathcal{M}^\#$ is decidable, then $(\mbor)^\#$ is decidable.

We can now obtain a partial converse to the previous corollary, which is one of the main results of this paper:

\begin{thm}
$\cal M$ has a decidable presentation if and only if $\mbor$ has an aware computable presentation.
\end{thm}

\begin{proof}
If $\mathcal M$ has a decidable presentation, then the induced presentation on $\mbor$ is aware by Lemma \ref{inducedaware} and computable by Corollary \ref{maincor2}.

Conversely, suppose that $(\mbor)^\#$ is a computable aware presentation of $\mbor$.  By Proposition \ref{awarewlog}, we may assume that $(\mbor)^\#$ is the presentation induced by a presentation $\mathcal{M}^\#$ of $\mathcal{M}$.  In order to show that $\mathcal{M}^\#$ is decidable, fix an arbitrary formula $\varphi(\vec x)$ and tuple $\vec a$ of generated points of $\mathcal{M}^\#$.  Since $\vec a$ can also be construed as a tuple of generated points of $(\mbor)^\#$ (and whose $(\mbor)^\#$-indices can be computed from their $\mathcal{M}^\#$-indices), we can approximate $\mu\llbracket\varphi(\vec a)\rrbracket$ to within $1/2$; since the value of $\mu\llbracket\varphi(\vec a)\rrbracket$ is either $1$ or $0$, depending on whether or not $\mathcal{M}\models \varphi(\vec a)$, this allows us to decide the latter statement.
\end{proof}

The proof of the second half of the preceding corollary ultimately relied on the computable categoricity of $\BorI$.  The reader might appreciate an argument which avoids this result, which goes as follows.  Maintain the notation in the previous proof.  Recall that we have a c.e. enumeration $B(f_k;r_k)$ of all of the rational open balls intersecting $\mathcal{M}$.  Suppose that each constant element $a_k$ was first found as being contained in the rational open ball $B(f_{i_k};r_{i_k})$.  Now for each $k=1,\ldots,n$, find $j_k\in \bb N$ such that $r_{j_k}$ is sufficiently small and for which $B(f_{j_k};r_{j_k})$ is formally included in $B(f_{i_k},r_{i_k})$; note that $a_k$ is the unique constant function in $B(f_{j_k};r_{j_k})$.  By assumption, one can approximate $\mu\llbracket \varphi(f_{j_1},\ldots,f_{j_k})\rrbracket$ arbitrarily well; if each $r_{j_1},\ldots,r_{j_n}$ is sufficiently small, then there is a large measure set of $\omega\in [0,1)$ for which $f_{j_k}(\omega)=a_k$ for all $k=1,\ldots,n$; consequently, $\mu\llbracket \varphi(f_{j_1},\ldots,f_{j_k})\rrbracket$ should be close to $1$ or $0$, allowing one to decide whether or not $\mathcal{M}\models \varphi(\vec a)$.

The following question is related to Question \ref{awarequestion} above:

\begin{question}
Can one drop the ``awareness'' assumption in the previous theorem?
\end{question}

One idea we had concerning the previous question was to find structures $\mathcal M$ and $\mathcal N$ such that $\mathcal M$ has a decidable presentation, $\mathcal N$ has a computable presentation but no decidable presentation, and satisfying $\mbor\cong \mathcal{N}^{[0,1)}$.  In this case, $\mathcal{N}^{[0,1)}$ would have a computable presentation, showing that one could not drop the awareness assumption in the previous theorem.  However, this attempt is doomed to fail as one can show that $\mbor\cong \mathcal{N}^{[0,1)}$ implies that $\mathcal{M}\cong \mathcal{N}$.  Indeed, if $\varphi$ is the Scott sentence for $\mathcal{M}$, then given that Borel randomizations and $L_{\omega_1,\omega}$ formulae interact properly (see \cite{keisler2}), from $\mbor\cong \mathcal{N}^{[0,1)}$ and the fact that $\mbor\models \llbracket\varphi\rrbracket$, one concludes that $\mathcal{N}^{[0,1)}\models \llbracket\varphi\rrbracket$ and thus $\mathcal{N}\models \varphi$, whence $\mathcal{M}\cong \mathcal{N}$.

In certain circumstances, any computable presentation of $\mbor$ is automatically aware.  We isolate one such situation here:

\begin{defn}
    We say that $\mathcal{M}$ is \textbf{effectively recognizable} if there is a presentation $\mathcal{M}^\#$ of $\mathcal{M}$ for which there is an algorithm such that, upon input a generated point $a$ of $\mathcal{M}^\#$ returns an $L$-formula $\varphi_a(x)$ such that $\varphi_a(\mathcal{M})=\{a\}$.
\end{defn}
\begin{examples}

\

    \begin{enumerate}
        \item Any computably presentable real closed field is effectively recognizable.
        \item For any $L$-structure $\mathcal{M}$, its expansion $\mathcal{M}_M$ obtained by naming all constants in $\mathcal{M}$ is effectively recognizable.
    \end{enumerate}
\end{examples}
\begin{lem}
    If $\mathcal{M}$ is effectively recognizable as witnessed by the presentation $\mathcal{M}^\#$, then the presentation $(\mbor)^\#$ is aware if it is computable.
\end{lem}
\begin{proof}
    Suppose that $B(f;\epsilon)$ is a rational open ball and $a$ is a generated point of $\mathcal{M}$ with corresponding formula $\varphi_a(x)$.  One then proceeds to approximate $\mu\llbracket \varphi_a(\vec f)\rrbracket$; if one finds that this measure exceeds $1-\epsilon$, this implies that $f(\omega)=a$ for a set of $\omega$'s of measure at least $1-\epsilon$, that is, $d(f,a)<\epsilon$, whence one returns this ball $B(f;\epsilon)$ as intersecting $\mathcal{M}$.
\end{proof}

\section{Computable categoricity of certain randomizations}\label{appendix2}

In this section, we prove that $\mbor$ is computably categorical when $\cal M$ is ``effectively $\omega$-categorical'', a notion that we define below.  As a warm-up, we treat the special case that $\cal M$ is the two-element set with no extra structure, which yields that $\BorI$, is computably categorical, a result of interest in its own right.  We note that this result can be proven by modifying the techniques in \cite{CMS}; here we give a direct proof that will generalize to a larger class of randomizations.   

Suppose that $\vec A=(A_1,...,A_m)$ is a tuple from $\BorI$ and $p(x)\in S_1(\vec A)$.  By quantifier-elimination for atomless probability algebras \cite[Proposition 16.6]{BBHU}, setting $n=2^m$, there are $r_1,...,r_n\in[0,1]$, such that, for any $C\in \BorI$, we have that $C$ realizes $p(x)$ if and only if: 
    $$\max_{i=1,...,n}\{|\mu(C\cap A_1^{i(1)}\cap...\cap A_m^{i(m)})-r_i|\}=0.$$  Here, we interpret $A^1=A$ and $A^0=[0,1)\setminus A$.

Set $\psi_p(x,\vec A)$ to be the formula
$$\max_{i=1,...,n}\{|\mu(x\cap A_1^{i(1)}\cap...\cap A_m^{i(m)})-r_i|\},$$ so $C\in \BorI$ realizes $p(x)$ if and only if $\psi_p(C,\vec A)=0$; in other words, $\psi_p$ isolates $p$.  We first show that $\psi_p(x,\vec A)$ isolates $p$ in an effective way.

\begin{lem}\label{NearRealization}
   For any $\epsilon>0$ and $B\in \BorI$ for which $\psi_p(B, \vec A)<\dfrac{\epsilon}{n}$, there is $B'\in\BorI$ such that $B'$ realizes $p$ and $d(B,B')<\epsilon$. 
\end{lem}
\begin{proof}
        Suppose that $\psi_p(B, \vec A)<\dfrac{\epsilon}{n}$.
    For each $i=1,...,n$, do the following:
    \begin{itemize}
        \item[i)] If $\mu(B\cap A_1^{i(1)}\cap...\cap A_m^{i(m)})-r_i>0$, set 
        $$B_i'=(B\cap A_1^{i(1)}\cap...\cap A_m^{i(m)})\backslash D_i,$$
        where $D_i\subseteq B\cap A_1^{i(1)}\cap...\cap A_m^{i(m)}$ and $\mu(D_i)=\mu(B\cap A_1^{i(1)}\cap...\cap A_m^{i(m)})-r_i.$ 
        \item[ii)] If $\mu(B\cap A_1^{k_1(i)}\cap...\cap A_m^{k_m(i)})-r_i<0$, set 
        $$B_i'=(B\cap A_1^{i(1)}\cap...\cap A_m^{i(m)})\cup D_i,$$
        where $D_i\subseteq  A_1^{i(1)}\cap...\cap A_m^{i(m)}$, $D_i\cap(B\cap A_1^{i(1)}\cap...\cap A_m^{i(m)})=\emptyset$ and $\mu(D_i)=r_i-\mu(B\cap A_1^{i(1)}\cap...\cap A_m^{i(m)}).$
        \item[iii)] If $\mu(B\cap A_1^{i(1)}\cap...\cap A_m^{i(m)})-r_i=0$, set
        $$B_i'=B\cap A_1^{i(1)}\cap...\cap A_m^{i(m)}.$$
    \end{itemize}
    Set $B'=\displaystyle\bigcup_{i=1}^nB_i'\in\BorI$.  By construction, $B'$ realizes $p$. Furthermore, note that
    $$ d(B,B')=\sum_{i=1}^nd\left(B\cap A_1^{i(1)}\cap...\cap A_m^{i(m)}, B_i'\right)<\epsilon.\qedhere$$
    \end{proof}
\begin{lem}\label{realizationsce}
    Suppose that $\BorI^\#$ is a computable presentation of $\BorI$, $\vec A$ a tuple of $\BorI^\#$-computable elements of $\BorI$, and $p(x)\in S_1(\vec A)$.  Further suppose that the corresponding real numbers $r_1,\ldots,r_n$ are computable real numbers.  Then the set of realizations of $p$ in $\BorI$ is a c.e. closed subset of $\BorI^\#$.  Moreover, an index for the set of realizations of $p$ can be computed from indices for $A_1,\ldots,A_m$ and $r_1,\ldots,r_n$.
\end{lem}
\begin{proof}
Let $\mathcal{G}$ denote the set of generated points of $\BorI^\#$. Let $\Omega$ denote the set of realizations of $p$, so $\Omega=\{E\in\BorI\ : \ \psi_p(E,\vec A)=0\}.$  We now describe an algorithm which enumerates all pairs $(C,\epsilon)$, where $C\in\mathcal G$ and $\epsilon>0$ is rational, for which $B(C,\epsilon)\cap\Om\neq\emptyset$.

We first note that, since $\vec A$ is computable, given rational $\delta>0$, one can compute a tuple of generated points $\vec A_\delta$ such that $d(\vec A,\vec A_\de)<\Delta_{\psi_p}(\de)$, and thus, for every $C\in\mathcal G$, we have $|\psi_p(C,\vec A)-\psi_p(C,\vec A_\de)|<\de.$

Here is the desired algorithm:  for each pair $(C,\epsilon)$ as above, we infinitely often return to considering this pair by searching for $D\in \mathcal G$ and rational $\delta>0$ such that $d(C,D)<\epsilon$ and $\psi_p(D,\vec A_\de)+\de<(\epsilon-d(C,D))/2^{|\vec A|}.$  Note that we can indeed effectively approximate $\psi_p(D,\vec A_\delta)$ since the presentation is computable.

First suppose we find such $D$ and $\delta$; we need to know that $B(C,\epsilon)\cap \Omega\not=\emptyset$.  We know that $\psi_p(D,\vec A)<(\epsilon-d(C,D))/2^{|\vec A|}$ and thus, by Lemma \ref{NearRealization}, there is $D'\in \Omega$ such that $d(D,D')<\epsilon-d(C,D)$; it follows that $$d(C,D')\leq d(C,D)+d(D,D')<\epsilon$$ and so $D'\in B(C,\epsilon)\cap \Omega$.

We now need to verify that if $B(C,\epsilon)\cap \Omega\not=\emptyset$, then this algorithm will eventually let us know this fact.  Take $E\in B(C,\epsilon)\cap \Omega$.  Since $\psi_p(E,\vec A)=0$, there is $D\in\mathcal G$ such that $d(C,D)<\epsilon$ and $\psi_p(D,\vec A)<\eta/2^{|\vec A|}$ for rational $\eta<\epsilon-d(C,D)$; it follows that, for sufficientlly small $\delta$, that $\psi_p(D,\vec A_\delta)+\delta<\epsilon-d(C,D)$ and our search will succeed.

Note finally that the uniformity statement holds by the description of the algorithm.
\end{proof}

The following is a general fact from effective metric structure theory (see \cite[Lemma 1.10]{UHF}):

\begin{fact}\label{cefact}
Suppose that $\mathcal N^\#$ is a computable presentation of a metric structure $\mathcal N$.  Then every nonempty c.e. closed subset $X$ of $\mathcal N^\#$ contains a computable point; moreover, an index for the computable point can be computed from an index for $X$. 
\end{fact}
\begin{thm}\label{BorIcompcat}
      $\BorI$ is computably categorical.
\end{thm}

\begin{proof} 
Fix computable presentations $\BorI^\#$ and $\BorI^\dagger$ of $\BorI$.  Let $(A_n)$ and $(B_n)$ be effective enumerations of the generated points of $\BorI^\#$ and $\BorI^\dagger$ respectively.  We effectively build a sequence $(\phi_n)$ of partial elementary maps from $\BorI$ to $\BorI$, where each domain is finite and such that:
\begin{itemize}
    \item For each $n$, $\phi_n\subseteq \phi_{n+1}$.
    \item If $n=2s$, then $A_s$ is in the domain of $\phi_n$.
    \item If $n=2s+1$, then $B_s$ is in the range of $\phi_n$.
    \item The domain of $\phi_n$ is a finite tuple of $\BorI^\#$-computable points, and the codes of these points can be computed from $n$.
    \item The range of $\phi_n$ is a finite tuple of $\BorI^\dagger$-computable points, and the indices of these points can be computed from $n$.
\end{itemize}
Suppose that $\phi_{n-1}$ has been defined.  We treat only the case $n=2s$, the other case being handled similarly.  Let $\{C_1,\ldots,C_{n-1}\}$ denote the domain of $\phi_{n-1}$.  Let $p_0(x)=\operatorname{tp}(A_n/\vec C)$ and let $p(x):=\phi_{n-1}(p_0)$; in other words, $p\in S_1(\phi_{n-1}(\vec C))$ is the type determined by the formula $\psi_p(x)$ obtained by replacing each occurrence of $C_i$ with $\phi_{n-1}(C_i)$; since $\phi_{n-1}$ is partial elementary, $p$ is indeed a consistent type.  Note also that the real numbers in $\psi_p$ are computable, and their indices can be computed from $n$ and indices for $C_1,\ldots,C_{n-1}$, which, by induction, are computable from $n-1$.  By countable saturation, $p$ is realized in $\BorI$.  By Lemma \ref{realizationsce} and Fact \ref{cefact}, we can compute the index of a $\BorI^\dagger$-computable point $D$ that realizes $p$, and we define $\phi_n:=\phi_{n-1}\cup\{(A_n,D)\}$.

Set $\Phi_0:=\bigcup_n \phi_n$.  Then $\Phi_0$ is a partial elementary map from $\BorI$ to $\BorI$ with dense domain and range.  It follows that $\Phi_0$ extends to an automorphism $\Phi$ of $\BorI$.  Moreover, by construction, the $\Phi$-image of each $\BorI^\#$-rational point is uniformly $\BorI^\dagger$-computable, whence $\Phi$ is a computable isomorphism.
\end{proof}


We now proceed to generalize this argument to certain randomizations $\mbor$.  Fix a Borel randomization $\mbor$, a tuple $\vec f=(f_1,\ldots,f_n)\in \cal K$, and a 1-type $p\in S_1(\vec f)$.  By quantifier-elimination for randomizations (see the next section for an effective version), the type $p(\vec X)$ is determined by the values of the formulae $\mu\llbracket \varphi(X,\vec f)\rrbracket$ as $\varphi(x,y_1,\ldots,y_n)$ ranges over all $L$-formulae.  In general, this is an infinite amount of information and an argument as in the case of $\BorI$ is not possible.

However, if we assume that $\cal M$ is $\omega$-categorical, then we can reduce this amount of information to a single formula.  Indeed, let $\theta_i(x,\vec y)$, $i=1,\ldots,m$, isolate the finitely many elements of $S_{n+1}(\operatorname{Th}(\cal M))$.  Then $p$ is determined by the formulae $\mu\llbracket \theta_i(X,\vec f)\rrbracket$ for $i=1,\ldots,m$.\footnote{This is clear for actual random variables, and the case of a general tuple $\vec f$ from $\cal K$ follows by continuity.}  In other words, there are real numbers $r_1,\ldots,r_m$ such that, setting $\psi_p(X,\vec Y)$ to be $$\max_{i=1,\ldots,m}\mu |\llbracket\theta_i(X,\vec Y)\rrbracket-r_i|,$$ we have that $g\in \cal K$ realizes $p$ if and only if $\psi_p(g,\vec f)=0$.  Moreover, just as in the case of $\BorI$, we have:

\begin{lem}
For any $\epsilon>0$ and $g\in \cal K$, if $\psi_p(g,\vec f)<\frac{\epsilon}{m}$, then there is $g'\in \cal K$ such that $g'$ realizes $p$ and $d(g,g')<\epsilon$.\footnote{The proof in the case that $g$ is a random variable mimics that of the proof of Lemma \ref{NearRealization}; the proof of the general case follows by approximating an arbitrary element of $\cal K$ by a random variable.}
\end{lem}

\begin{lem}
    Suppose that $(\mbor)^\#$ is a computable presentation of $\mbor$, $\vec f$ a tuple of $(\mbor)^\#$-computable elements of $\mbor$, and $p(X)\in S_1(\vec f)$.  Further suppose that the corresponding real numbers $r_1,\ldots,r_m$ are computable real numbers.  Then the set of realizations of $p$ in $\mbor$ is a c.e. closed subset of $(\mbor)^\#$.  Moreover, an index for the set of realizations of $p$ can be computed from indices for $f_1,\ldots,f_m$ and $r_1,\ldots,r_m$.
\end{lem}

In order to carry out the back-and-forth argument as in the proof of Theorem \ref{BorIcompcat}, we need to add one extra condition to $\cal M$:

\begin{defn}\label{effomega}
Let $\cal M$ be a classical $\omega$-categorical structure.  We say that $\cal M$ is \textbf{effectively $\omega$-categorical} if there is an algorithm such that, upon input $n\in \mathbb N$, returns a finite list of formulae $\theta_1,\ldots,\theta_m$ that isolate the finitely many elements of $S_{n+1}(T)$.
\end{defn}

\begin{example}
If $\cal M$ is a structure in a finite relational language with quantifier elimination, then $\cal M$ is effectively $\omega$-categorical.
\end{example}

The reason we need to assume that $\cal M$ is effectively $\omega$-categorical is the following.  Work in the notation of the proof of Theorem \ref{BorIcompcat} (but adapted to the context of randomizations).  We have at this point constructed a partial elementary map $\phi_{n-1}$ with domain $\{f_1,\ldots,f_{n-1}\}$, a finite set of $(\mbor)^\#$-computable points whose indices can be computed from $n-1$; we wish to add a $(\mbor)^\#$-generated point $g$ to the domain.  We consider $p_0(X):=\operatorname{tp}(g/\vec f)$ and set $p(X):=\phi_{n-1}(p_0)$.  In order to find a $(\mbor)^\dagger$ code for the set of realizations of $p$, we need to compute the formula that isolates $p_0$, that is, we need to be able to compute the finite set of formulae isolating the relevant set of types in $\operatorname{Th}(\cal M)$.  This is precisely what effective $\omega$-categoricity affords us.  Thus, we may repeat the proof of Theorem \ref{BorIcompcat} mutatis mutandis, obtaining:

\begin{thm}
Suppose that $\cal M$ is an effectively $\omega$-categorical theory.  Then $\mbor$ is computably categorical.
\end{thm}

We should point out that, in the context of the previous theorem, $\mbor$ does indeed have a computable presentation, whence the theorem is not vacuous.  Indeed, if $\cal M$ is $\omega$-categorical, then it has a decidable presentation, whence the induced presentation on $\mbor$ is computable by Corollary \ref{maincor} above.

\section{Effective quantifier-elimination for randomizations}\label{appendix1}

In this section, we prove that every theory of randomizations (in a computable language) admits effective quantifier elimination.  Since effective quantifier elimination seems not to have been considered before in continuous logic, we begin by defining the notion precisely and establishing some preliminary lemmas.

\begin{defn}
Suppose that $T$ is a continuous theory in a computable language.  We say that $T$ has \textbf{effective quantifier-elimination} if there is an algorithm such that, upon input a restricted formula $\varphi(\vec x)$ and rational $\epsilon>0$, returns a quantifier-free restricted formula $\psi(\vec x)$ such that $\|\varphi-\psi\|_T<\epsilon$.
\end{defn}

Here, given arbitrary formulae $\varphi$ and $\psi$ in the language for $T$, we define $$\|\varphi-\psi\|_T:=\sup\{|\varphi(\vec a)^{\cal N}-\psi(\vec a)^{\cal N}| \ : \ \cal N\models T, \ \vec a\in \cal N\}.$$

Our goal in this section is to prove the following theorem:

\begin{thm}\label{randomQE}
Suppose that $L$ is a classical computable language.  Then for any $L$-structure $\cal M$, the complete theory of $\mbor$ has effective quantifier-elimination via an algorithm that depends only on $L$ (and not on $\cal M$). 
\end{thm}

We begin by developing a useful test for effective quantifier-elimination.  

\begin{lem}\label{QEtest}
Suppose that there is an algorithm such that, upon input a restricted quantifier-free formula $\varphi(x,\vec y)$ and $\epsilon>0$, returns a restricted quantifier-free formula $\psi(\vec y)$ such that $$\|\inf_x \varphi(x,\vec y)-\psi(\vec y)\|_T<\epsilon.$$  Then $T$ admits effective quantifier-elimination.
\end{lem}

\begin{proof}
Suppose a restricted formula $\varphi(\vec x)$ and rational $\epsilon>0$ are given.  First note that the procedure for converting a restricted formula into an equivalent restricted formula in prenex normal form is effective (see \cite[Proposition 6.9]{BBHU}), whence we may assume that $\varphi$ is itself in prenex normal form.  Write $$\varphi(\vec x)=Q_1y_1\cdots Q_my_m \psi(\vec x,\vec y)$$ with $\psi(\vec x,\vec y)$ a quantifier-free restricted formula.  By assumption, we may find a quantifier-free restricted formula $\psi_1(\vec x,y_1,\ldots,y_{m-1})$ such that $$\|Q_my_m\psi(\vec x,y_1,\ldots,y_{m-1},y_m)-\psi_1(\vec x,y_1,\ldots,y_{m-1})\|_T<\frac{\epsilon}{m}.$$  Similarly, one can find a quantifier-free restricted formula $\psi_2(\vec x,y_1,\ldots,y_{m-2})$ such that $$\|Q_{m-1}y_{m-1}\psi_1(\vec x,y_1,\ldots,y_{m-2},y_{m-1})-\psi_2(\vec x,y_1,\ldots,y_{m-2})\|_T<\frac{\epsilon}{m}.$$  Continuing in this way, we obtain a sequence of quantifier-free restricted formula $\psi_1(\vec x,y_1,\ldots,y_{m-1}),\ldots,\psi_m(\vec x)$.  It follows that $\|\varphi(\vec x)-\psi_m(\vec x)\|_T<\epsilon$.    
\end{proof}

Our next goal is to establish the fact that atomless probability algebras admit effective quantifier elimination.  Before doing so, we need a few preparatory results.  The first one is well-known and follows from the proof of the corresponding classical fact:

\begin{lem}\label{compmod}
There is an algorithm such that, upon input $n\geq 1$ and a code for a computable continuous function $u:[0,1]^n\to [0,1]$, returns a code for a modulus of uniform continuity for $u$.
\end{lem}

\begin{lem}\label{compresetricted}
There is an algorithm such that, upon input a computable continuous function $u:[0,1]^n\to [0,1]$ and rational $\epsilon>0$, returns a restricted function $v:[0,1]^n\to [0,1]$ such that $\|u-v\|_\infty<\epsilon$.
\end{lem}

\begin{proof}
Given a restricted function $v$, find $n\in \mathbb N$ such that $$2^{-n}<\min(\Delta_u(\epsilon/6),\Delta_v(\epsilon/6)).$$  For each $k=1,\ldots,2^n$, find open intervals $I_k,J_k\subseteq [0,1]$ with rational endpoints of length at most $\epsilon/6$ such that $u(k/2^n)\in I_k$ and $v(k/2^n)\in J_k$.  The algorithm then checks if $$\max\{|s-s'| \ : \ s\in I_k,s'\in J_k\}<\epsilon/2; \quad (\dagger)$$ if the answer is yes, then the algorithm returns $\nu$ as the desired restricted function.  To see that $\nu$ is as claimed, note that, for any $x\in [k/2^n,(k+1)/2^n)$, we have
$$|u(x)-v(x)|\leq |u(x)-u(k/2^n)|+|u(k/2^n)-v(k/2^n)|+|v(k/2^n)-v(x)|\leq 5\epsilon/6<\epsilon.$$  Also note that, if $\nu$ is such that $\|u-v\|_\infty<\epsilon/6$, then indeed $(\dagger)$ above holds, so the algorithm always terminates.
\end{proof}

For a continuous function $u:[0,1]^n\to [0,1]$, define $\min(u):[0,1]^n\to [0,1]$ by $\min(u)(r_1,\ldots,r_n):=\min\{u(s_1,\ldots,s_n) \ : \ s_i\leq r_i \ i=1,\ldots,n\}$.  Then $\min(u)$ is itself a continuous function.  Moreover, we have the following result, whose straightforward proof we leave to the reader:

\begin{lem}\label{compmin}
If $u$ is computable, then so is $\min(u)$.  Moreover, there is an algorithm such that, upon input the code for a computable continuous function $u:[0,1]^n\to [0,1]$, returns a code for $\min(u)$.
\end{lem}

\begin{thm}
The theory of atomless probability algebras admits effective quantifier-elimination.
\end{thm}

\begin{proof}
Let $T$ denote the theory of atomless probability algebras.  We apply the test described in Lemma \ref{QEtest}.  Suppose that $\varphi(x,\vec y)$ and $\epsilon>0$ are given.  Write $\vec y=(y_1,\ldots,y_n)$.  Without loss of generality, we may write 
$$\varphi(x,\vec y)=u((\mu(x\cap \vec y^k)_{k\in 2^n})),$$ where $u:[0,1]^{2^n}\to [0,1]$ is a restricted function and $\vec y^k:=y_1^{k(1)}\cap \cdots \cap y_n^{k(n)}$.  We then have that $\inf_x\varphi(x,\vec y)$ is equivalent to $\min(u)((\mu(\vec y^k)_{k\in 2^n}))$.  By Lemma \ref{compmin}, we can compute a code for $\min(u)$ from a code for $u$; by Lemma \ref{compresetricted}, we can then compute a restricted function $v$ such that $\|\min(u)-v\|_\infty<\epsilon$.  It follows that $\|\inf_x\varphi(x,\vec y)-v((\mu(\vec y^k))_{k\in 2^n})\|_T<\epsilon$.
\end{proof}

The other ingredient needed in the proof of Theorem \ref{randomQE} is an appropriate notion of a family of definable sets in continuous logic.  To the best of our knowledge, such a definition has not appeared in the literature thus far.

\begin{defn}
Suppose that $\cal N$ is a metric structure.  A formula $\varphi(\vec x,\vec y)$ defines a \textbf{definable family of definable subsets of $\cal N$} if there is a function $\Delta:(0,1)\to (0,1)$ such that, for all $\vec a,\vec b\in \cal N$ and $\epsilon>0$, if $\varphi(\vec a,\vec b)<\Delta(\epsilon)$, then there is $\vec a'\in \cal N$ such that $d(\vec a,\vec a')<\epsilon$ and $\varphi(\vec a',\vec b)=0$.
\end{defn}

In what follows, given a formula $\varphi(\vec x,\vec y)$ and $\vec b\in \cal N$, we denote the zeroset of $\varphi(\vec x,\vec b)$ by $Z(\varphi(\vec x,\vec b)):=\{\vec a\in \cal N \ : \ \varphi(\vec a,\vec b)\}=0$. 

\begin{prop}\label{defprop}
Suppose that $\varphi(\vec x,\vec y)$ defines a definable family of definable subsets of $\cal N$.  Then for any formula $\psi(\vec x,\vec y,\vec z)$, the predicate 
$$P(\vec b,\vec c):=\inf\{\psi(\vec a,\vec b,\vec c) \ : \ \varphi(\vec a,\vec b)=0\}$$ is a definable predicate in $\cal N$.  Moreover:
\begin{enumerate}
    \item If $\varphi$ and $\psi$ are quantifier-free, then $P$ can be uniformly approximated (in $\cal N$) by formulae of the form $\inf_{\vec w}\chi(\vec x,\vec y,\vec z)$ with $\chi$ quantifier-free and $|\vec w|=2|\vec x|$.
    \item If the modulus $\Delta$ witnessing that $\varphi$ defines a definable family of definable subsets of $\cal N$ is computable, then there is an algorithm such that, upon input a restricted quantifier-free formula $\psi(\vec x,\vec y,\vec z)$ and rational $\epsilon>0$, returns a restricted quantifier-free formula $\chi(\vec x,\vec y,\vec z)$ such that $$\|P(\vec b,\vec c)-\inf_{\vec x}\chi(\vec x,\vec b,\vec c)^{\cal N}\|<\epsilon$$ for all $\vec b,\vec c\in \cal N$.  Moreover, a code for this algorithm can be computed from a code for $\Delta$.
\end{enumerate}
\end{prop}

\begin{proof}
We follow the proofs of \cite[Theorem 9.17 and Proposition 9.19]{BBHU}, just noting that they hold in this more general setting and that they hold effectively.  Take a continuous, non-decreasing function $\alpha:[0,1]\to [0,1]$ with $\alpha(0)=0$ such that $|\psi(\vec a,\vec b)-\psi(\vec a',\vec b)|\leq \alpha(d(\vec a,\vec a'))$ for all $\vec a,\vec a',\vec b\in \cal N$.  Note that if $\psi$ is restricted, then $\alpha$ is computable and a code for $\alpha$ can be computed from a code for the modulus of uniform continuity of $\psi$ (see \cite[Section 2]{goldhart}), which in turn can be computed from a code for $\psi$ by Lemma \ref{compmod} above.  Since $\varphi$ defines a definable family of definable subsets of $\cal N$, there is a continuous, non-decreasing function $\beta:[0,1]\to [0,1]$ such that $\beta(0)=0$ and such that, for all $\vec a,\vec b\in \cal N$, one has $d(\vec a,Z(\varphi(\vec x, \vec b)))\leq \beta(\varphi(\vec a,\vec b))$.  As in the previous sentence, if if $\Delta$ is computable, then a code for $\beta$ can be computed from a code for $\Delta$.  

Consider the formula $\Phi(\vec x,\vec y)=\inf_{\vec x'}\min(\beta(\varphi(\vec x',\vec y))+d(\vec x,\vec x'),1)$.  Then for any $\vec b\in \cal N$, one has $\Phi(\vec a,\vec b)=d(\vec a,Z(\varphi(\vec x,\vec b)))$.   Finally, set $$\zeta(\vec y,\vec z)=\inf_{\vec x}\left(\psi(\vec x,\vec y,\vec z)+\alpha(\Phi(\vec x,\vec y))\right).$$  One checks as in \cite[Theorem 9.17]{BBHU} that $P(\vec b,\vec c)=\zeta(\vec b,\vec c)^\cal N$ for all $\vec b,\vec c\in \cal N$.  

Setting $\Lambda(\vec x,\vec x',\vec y)$ to be the quantifier-free portion of $\Phi$, we have that $\Phi(\vec x,\vec y)$ is equivalent to the formula
$$\inf_{\vec x}\inf_{\vec x'}(\psi(\vec x,\vec y,\vec z)+\alpha(\Lambda(\vec x,\vec x',\vec y)).$$

As a result, if $\varphi$ and $\psi$ are quantifier-free, then $\zeta$ is an existential formulae with $2|\vec x|$ many existential quantifiers.  If $\beta$ is computable and $\varphi$ is quantifier-free, then by Lemma \ref{compresetricted}, one can effectively approximate $\Lambda$ by restricted quantifier-free formulae. Another application of Lemma \ref{compresetricted} shows that if $\psi$ is also quantifier-free and $\alpha$ is computable, then $\zeta$ can be effectively approximated by restricted quantifier-free formulae, uniformly in codes for $\alpha$ and $\psi$, completing the proof of the proposition.  
\end{proof}

We now describe the particular family of definable sets that concern us in our proof of effective quantifier-elimination for randomizations.  The fact motivating the family we describe below is the following:

\begin{fact}\label{WitnessPartition}
    Let $\theta_i(x,\vec{y})$, $i=1,...,n$ be a sequence of $L$-formulas such that 
    $$\models(\theta_1\vee...\vee\theta_n)\text{ and }\models\neg(\theta_i\wedge\theta_j)\text{ for }1\leq i<j\leq n.$$
    Consider a Borel randomization $\mbor$ of an $L$-structure $\cal M$.  Fix $B_1,...,B_n\in\mathcal{B}$ and random variables $\vec{f}\in \mathcal{K}$. Then the following are equivalent:
    \begin{itemize}
        \item[i)] There exists a random variable $g\in \cal K$ such that, for each $i=1,\ldots n$, we have
        $$B_i=\llbracket\theta_i(g,\vec{f})\rrbracket.$$
        \item[ii)] The events $B_1,\ldots,B_n$ form a partition of $\top$ and, for each $i=1,\ldots,n$, we have
        $$ B_i\subseteq\llbracket\exists x\theta_i(x,\vec{f})\rrbracket.$$
    \end{itemize}
\end{fact}

For $\theta_1,\ldots,\theta_n$ as in the previous fact, consider the formula $\Phi(\vec{A},\vec{Y})$, where $\vec A$ is a tuple of variables in sort $\mathcal B$ and $\vec Y$ is a tuple of variables in sort $\mathcal{K}$, defined by $$\max\left\{\displaystyle\sum_{i=1}^n\mu(A_i)\dot-\mu\left(\bigcup_{i=1}^nA_i\right),1\dot-\mu\left(\bigcup_{i=1}^nA_i\right),\left\{\mu(A_i)\dot-\mu(A_i\cap\llbracket\exists x\theta_i(x,\vec{Y})\rrbracket)\right\}_{i=1}^n \right\}.$$
Note that, for a given tuple of random variables $\vec f$ from $\cal K$, the zeroset of $\Phi(\vec A,\vec f)$ are precisely those tuples $\vec B$ satisfying (ii) in the previous fact.

 \begin{lem}\label{ClaimDef}
       For any $L$-structure $\cal M$, the formula $\Phi(\vec{A},\vec{Y})$ defines a definable family of definable subsets of $\mbor$.  Moreover, the modulus for $\Phi$ is computable and independent of the choice of $\cal M$.
    \end{lem}
    \begin{proof}
    Given $\epsilon>0$, we will show that that, for any tuple of events $\vec E$ and tuple of random variables $\vec f$, if $$\Phi(\vec E,\vec f)<n\epsilon/6(n+1)^2,$$ then $d\left(\vec E,Z(\Phi(\cdot,\vec f))\right)<\epsilon.$  We leave it to the reader to check that, in what follows, we may assume, without loss of generality, that $\vec f$ is a tuple of random variables; indeed, this follows from a simple continuity argument.

     \textbf{Step 1:} Fix $\delta>0$ and take $\eta<(n/n+1)\delta$.  For any $\eta$-partition $\vec E$, that is, a tuple $\vec E$ such that $$\max\left(\sum_{i=1}^m\mu(E_i)\dot-\mu(\bigcup_{i=1}^m E_i),1\dot-\mu(\bigcup_{i=1}^m E_i)\right)<\eta,$$ there is a partition $\vec V$ such that $d(\vec E,\vec V)<\delta$.

     To establish step 1, let $N=2^n$ and for each $k=1,...,N$, consider all the possible events given by 
     $$E_1^{k(1)}\cap...\cap E_n^{k(n)},$$
     where each $k:\{1,...,n\}\to\{0,1\}$, $E_i^0=\neg E_i$, and $E_i^1=E_i$. For each $k=1,...,N,$ set 
     $$\Gamma(k)=\{i\in\{1,...,n\}\ : \ k(i)=1\}.$$
     If $\Gamma(k)\neq \emptyset$, let $\{V_i^k\}_{i\in\Gamma(k)}$ be a ``fair'' (measurable) partition of $E_1^{k(1)}\cap...\cap E_n^{k(n)}$, that is, for each $i\in \Gamma(k)$, we have 
     $$\mu(V_i^k)=\dfrac{\mu(E_1^{k(1)}\cap...\cap E_n^{k(n)})}{|\Gamma(k)|}.$$
     If $\Gamma(k)=\emptyset$, let $W_1,...,W_n$ be a fair partition of $E_1^0\cap...\cap E_n^0$. For each $i=1,\ldots, n$, define 
     $$V_i=\left(\bigcup_{\{k\  \ : \ i\in \Gamma(k)\}}V_i^k\right)\cup W_i.$$
     By construction, $\vec V$ form a partition of $\top$. We now show that $d(\vec E,\vec V)<\delta.$ First note that, by construction, for each $i=1,...,n$, we have $V_i\backslash W_i\subset E_i$ and $E_i\cap W_i=\emptyset $, and thus
     $$d(E_i,V_i)=\mu(E_i\backslash(V_i\backslash W_i))+\mu(W_i) $$
     $$ =\mu\left(E_i\backslash\bigcup_{\{k\ |\ i\in \Gamma(k)\}}V_i^k\right)+\mu(W_i).$$
     Since $1-\mu(\bigcup_{i=1}^n E_i)<\eta$, then $\mu(W_i)<\eta/n$. In the same way, since 
     $$E_i\backslash\bigcup_{\{k\ |\ i\in \Gamma(k)\}}V_i^k\subset\bigcup_{\ell=1,\ell\neq i}^n(E_i\cap E_\ell)\subset\bigcup_{i\neq j}(E_i\cap E_j)$$
     and $$\mu\left(\bigcup_{i=1}^n E_i\right)+\mu\left(\bigcup_{i\neq j}(E_i\cap E_j)\right)\leq \sum_{i=1}^n\mu(E_i),$$ we have $\mu\left(\bigcup_{i\neq j}^n(E_i\cap E_j)\right)<\eta$ and so $$\mu\left(E_i\backslash\bigcup_{\{k\ |\ i\in \Gamma(k)\}}V_i^k\right)<\eta.$$
     Consequently, 
     $$d(E_i,V_i)<\eta+\eta/n<\delta,$$
     and $d(\vec E,\vec V)<\delta$ as desired. 

     \textbf{Step 2:} Fix $\epsilon>0$, take $\delta<\epsilon/4(n+1)$ and $\eta<n/(n+1)\delta$.  Let $\vec E$ be such that $\Phi(\vec E,\vec f)<\eta$. Since $\vec E$ is an $\eta$-partition, by Step 1, there is a partition $\vec V$ such that $d(\vec E,\vec V)<\delta$. We now show how to find $\vec B$ such that $\Phi(\vec B,\vec f)=0$ and $d(\vec V,\vec B)< \epsilon/2$, whence $d(\vec E,\vec B)<\epsilon$, as desired.

     First note that, since for each $i=1,...,n$, $\mu(E_i)\dot-\mu(E_i\cap\llbracket\exists x\theta_i(x,\vec f)\rrbracket)<\eta$, we have:
     $$\mu(V_i\cap \llbracket\neg\exists x\theta_i(x,\vec f)\rrbracket)=\mu((V_i\cap E_i)\cap \llbracket\neg\exists x\theta_i(x,\vec f)\rrbracket)+\mu((V_i\setminus E_i)\cap \llbracket\neg\exists x\theta_i(x,\vec f)\rrbracket)<\eta+\delta.$$

     
     Now for each $i=1,...,n$, set 
     $$\Sigma(i)=\{ j\in\{1,...,n\}\ |\ (V_i\cap\llbracket\neg\exists x\theta_i(x,\vec f)\rrbracket)\cap\llbracket\exists x\theta_j(x,\vec f)\rrbracket\neq \emptyset\}.$$
     If $\Sigma(i)\neq\emptyset$, let $\{B_j^i\}_{j\in\Sigma(i)}$ be a partition of $V_i\cap \llbracket\neg\exists x\theta_i(x,\vec f)\rrbracket$ such that, for all $j\in \Sigma(i)$, we have $B_j^i\subseteq \llbracket\exists x\theta_j(x,\vec f)\rrbracket$.  For $k=1,\ldots,n$, set
     $$B_k:=(V_k\cap\llbracket\exists x\theta_k(x,\vec f)\rrbracket)\cup\bigcup_{\{i\ |\ k\in\Sigma(i)\}}B_k^i. $$
     By construction, $\Phi(\vec B,\vec f)=0$. Now, note that for each $j=1,...,n$, we have
     $$d(V_j,B_j)=\mu(V_j)\dot-\mu(V_j\cap\llbracket\exists x\theta_j(x,\vec f)\rrbracket)+\sum_{\{i\ |\ j\in\Sigma(i)\}}\mu(B_j^i) $$
     $$ <(\delta+\eta)+n(\delta+\eta)<2(n+1)\delta<\epsilon/2$$
     and thus $d(\vec V,\vec B)<\epsilon/2$, as desired. 
    \end{proof}



\begin{proof}[Proof of Theorem \ref{randomQE}]
We apply the test from Lemma \ref{QEtest}.  Towards that end, suppose we are given an existential restricted $L^R$-formula $\varphi(\vec X,\vec A)$, where $\vec X$ is a tuple of variables of sort $\cal K$ and $\vec A$ is a tuple from sort $\cal B$.  If $\varphi$ has the form $\inf_C \psi(\vec X,\vec A,C)$, where $C$ ranges over $\cal B$, then by (the proof of) effective quantifier-elimination for atomless probability algebras, we can effectively approximate $\varphi$ by quantifier-free restricted formulae.

We thus assume that $\varphi$ is an existential formula in $L^R$ of the form 
    $$\displaystyle\inf_X\psi_1(X,\vec{Y},\vec{A}),$$
    where $\psi_1$ is a quantifier-free formula in $L^R$ and $X,\vec{Y}$ are variables in sort $\mathcal{K}$ and $\vec{A}$ are variables in sort $\mathcal{B}$. The variable $X$ occurs in $\psi_1$ in finitely many terms $\llbracket\varphi_i(X,\vec{Y})\rrbracket$, $i\leq m$ of sort $\mathcal{B}$, where each $\varphi_i(x,\vec{y})$ is an $\mathcal{L}$ formula. Let $\theta_j(x,\vec{y})$ for $j\leq n=2^m$ be the list of all possible  conjunctions
    $$\varphi_1^{k(1)}\wedge...\wedge\varphi_m^{k(m)}$$
    where each $k:\{1,...,n\}\to\{0,1\}$, $\varphi_i^0=\neg\varphi_i$ and $\varphi_i^1=\varphi_i$ for all $i\leq m$. So using Boolean Axioms, the formula $\psi_1$ is equivalent  to a quantifier-free $\psi_2$ in which the variable $X$ occurs only in terms $\llbracket\theta_j(X,\vec{Y})\rrbracket$. Let $\psi_3(\vec Y,\vec A,\vec B)$ be quantifier-free formula obtained from $\psi_2$ by replacing each term $\llbracket\theta_j(X,\vec{Y})\rrbracket$ by a new variable $B_j$ of sort $\mathcal{B}$. 
    
    Note that the variable $X$ does no occurs in $\psi_3$, and now $\varphi$ is equivalent to 
    $$\inf\{\psi_3(\vec{Y},\vec{A},\vec{B}) \ : \ \Phi(\vec B,\vec Y)=0\}.$$  By Proposition \ref{defprop}, this latter expression is effectively approximable by formulae of the form $\inf_{\vec B}\chi(\vec Y,\vec A,\vec B)$ with $\chi$ a restricted quantifier-free formula.  However, we already know from the first case that such existential formulae can themselves be effectively approximated by quantifier-free restricted formulae, finishing the proof of the theorem.
\end{proof}

\end{document}